\documentclass[final,onefignum,onetabnum]{siamonline190516}

\usepackage{lipsum}
\usepackage{amsfonts}
\usepackage{graphicx}
\usepackage{epstopdf}
\usepackage{algorithmic}
\ifpdf
  \DeclareGraphicsExtensions{.eps,.pdf,.png,.jpg}
\else
  \DeclareGraphicsExtensions{.eps}
\fi
\usepackage{booktabs}
\usepackage{multirow}
\usepackage[T1]{fontenc}
\usepackage[final]{microtype}
\usepackage{enumitem}
\setlist[enumerate]{leftmargin=.5in}
\setlist[itemize]{leftmargin=.5in}

\newsiamremark{remark}{Remark}
\newsiamremark{hypothesis}{Hypothesis}
\crefname{hypothesis}{Hypothesis}{Hypotheses}
\newsiamthm{claim}{Claim}

\headers{Parareal Neural Networks Emulating a Parallel-in-time Algorithm}{C.-O. Lee, Y. Lee, and J. Park}

\title{Parareal Neural Networks Emulating a Parallel-in-time Algorithm\thanks{Submitted to the editors DATE.
\funding{The first author's work was supported by the National Research Foundation~(NRF) of Korea grant funded by the Korea government~(MSIT) (No.~2020R1A2C1A01004276).
The third author's work was supported by Basic Science Research Program through NRF funded by the Ministry of Education (No.~2019R1A6A1A10073887).}}}

\author{Chang-Ock Lee\thanks{Department of Mathematical Sciences, KAIST, Daejeon 34141, Korea 
  (\email{colee@kaist.edu}, \email{lyk92@kaist.ac.kr}).}
\and Youngkyu Lee\footnotemark[2]
\and Jongho Park \thanks{Natural Science Research Institute, KAIST, Daejeon 34141, Korea 
  (\email{jongho.park@kaist.ac.kr}).}}

\usepackage{amsopn}

\usepackage{amsmath,bm,mathtools,amssymb}
\graphicspath{ {./figure/} }
\newcommand\br{\mathbf{r}}
\newcommand\bu{\mathbf{u}}
\newcommand\bx{\mathbf{x}}
\newcommand\by{\mathbf{y}}
\newcommand\bS{\mathbf{S}}
\newcommand\bU{\mathbf{U}}
\newcommand\bdelta{\bm{\delta}}

\ifpdf
\hypersetup{
  pdftitle={Parareal Neural Networks Emulating a Parallel-in-time Algorithm},
  pdfauthor={C.-O. Lee, Y. Lee, and J. Park}
}
\fi

\begin{document}

\maketitle

% REQUIRED
\begin{abstract}
  As deep neural networks~(DNNs) become deeper, the training time increases. In this perspective, multi-GPU parallel computing has become a key tool in accelerating the training of DNNs.
  In this paper, we introduce a novel methodology to construct a parallel neural network that can utilize multiple GPUs simultaneously from a given DNN.
  We observe that layers of DNN can be interpreted as the time steps of a time-dependent problem and can be parallelized by emulating a parallel-in-time algorithm called parareal.
  The parareal algorithm consists of fine structures which can be implemented in parallel and a coarse structure which gives suitable approximations to the fine structures.
  By emulating it, the layers of DNN are torn to form a parallel structure, which is connected using a suitable coarse network.
  We report accelerated and accuracy-preserved results of the proposed methodology applied to VGG-16 and ResNet-$1001$ on several datasets.
\end{abstract}

% REQUIRED

% REQUIRED

\section{Introduction}
\label{Sec:Int}
Deep neural networks~(DNNs) have demonstrated success for many classification and regression tasks such as image recognition~\cite{he2016deep,lecun2015deep} and natural language processing~\cite{collobert2011natural,jean2014using}.
A principal reason for why DNN performs well is the depth of DNN, i.e., the number of sequential layers of DNN.
Each layer of DNN is composed of a combination of an affine transformation and a nonlinear activation function, e.g., a rectified linear unit~(ReLU).
A broad range of functions can be generated by stacking a number of layers with nonlinear activation functions so that DNN can be a model that fits the given data well~\cite{cybenko1989approximation,hornik1989multilayer}.
However, there are undesirable side effects of using many layers for DNN.
Due to the large number of layers in DNNs, DNN training is time-consuming and there are demands to reduce training time these days.
Recently, multi-GPU parallel computing has become an important topic for accelerating DNN training~\cite{ben2019demystifying,chen2018efficient,gunther2020layer}.

Data parallelism~\cite{ben2019demystifying} is a commonly used parallelization technique. In data parallelism, the training dataset is distributed across multiple GPUs and then processed separately.
For instance, suppose that we have 2~GPUs and want to apply data parallelism to the mini-batch gradient descent with the batch size $128$.
In this case, each GPU computes $64$-batch and the computed gradients are averaged.
In data parallelism, each GPU must possess a whole copy of the DNN model so that inter-GPU communication is required at every step of the training process in order to update all parameters of the model.
Therefore, the training time is seriously deteriorated when the number of layers in the model is large.
In order to resolve such a drawback, several asynchronous methodologies of data parallelism  were proposed~\cite{chen2016revisiting,lian2015asynchronous,zhang2016staleness}.
In asynchronous data parallelism, a parameter server in charge of parameter update is used; it collects computed gradients from other GPUs in order to update parameters, and then distributes the updated parameters to other GPUs.

On the other hand, model parallelism~\cite{huang2019gpipe,narayanan2019pipedream} is usually utilized when the capacity of a DNN exceeds the available memory of a single GPU.
In model parallelism, layers of DNN and their corresponding parameters are partitioned into multiple GPUs.
Since each GPU owns part of the model's parameters, the cost of inter-GPU communication in model parallelism is much less than the cost of data parallelism.
However, only one GPU is active at a time in the naive application of model parallelism.
To resolve the inefficiency, a pipelining technique called PipeDream~\cite{narayanan2019pipedream} which uses multiple mini-batches concurrently was proposed.
PipeDream has a consistency issue in parameter update that a mini-batch may start the training process before its prior mini-batch updates parameters.
To avoid this issue, another pipelining technique called Gpipe~\cite{huang2019gpipe}  was proposed; it divides each mini-bath into micro-batches and utilizes micro-batches for the simultaneous update of parameters.
However, experiments~\cite{chen2018efficient} have shown that the possible efficiency of Gpipe can not exceed 29\% of that of Pipedream.
Recently, further improvements of PipeDream and Gpipe were considered; see SpecTrain~\cite{chen2018efficient} and PipeMare~\cite{yang2019pipemare}.

There are several notable approaches of parallelism based on layerwise decomposition of the model~\cite{fok2018decoupling,gunther2020layer}.
Unlike the aforementioned ones, these approaches modify data propagation in the training process of the model.
G\"{u}nther et al.~\cite{gunther2020layer} replaced the sequential data propagation of layers in DNN by a nonlinear in-time multigrid method~\cite{falgout2014parallel}.
It showed strong scalability in a simple ResNet~\cite{he2016deep} when it was implemented on a computer cluster with multiple CPUs.
Fok et al.~\cite{fok2018decoupling} introduced WarpNet which was based on ResNet.
They replaced residual units~(RUs) in ResNet by the first-order Taylor approximations, which enabled parallel implementation.
In WarpNet, $(N-1)$ RUs are replaced by a single warp operator which can be treated in parallel using $N$ GPUs.
However, this approach requires data exchange at every warp operation so that it may suffer from a communication bottleneck as the DNN becomes deeper.

In this paper, we propose a novel paradigm of multi-GPU parallel computing for DNNs, called \textit{parareal neural network}.
In general, DNN has a feed-forward architecture.
That is, the output of DNN is obtained from the input by sequential compositions of functions representing layers.
We observe that sequential computations can be interpreted as time steps of a time-dependent problem.
In the field of numerical analysis, after a pioneering work of Lions et al.~\cite{lions2001resolution}, there have been numerous researches on parallel-in-time algorithms to solve time-dependent problems in parallel; see, e.g.,~\cite{gander2007analysis,maday2002parareal,minion2011hybrid}.
Motivated by these works, we present a methodology to transform a given feed-forward neural network to another neural network called parareal neural network which naturally adopts parallel computing.
The parareal neural network consists of fine structures which can be processed in parallel and a coarse structure which approximates the fine structures by emulating one of the parallel-in-time algorithms called parareal~\cite{lions2001resolution}.
Unlike the existing methods mentioned above, the parareal neural network can significantly reduce the time for inter-GPU communication because the fine structures do not communicate with each other but communicate only with the coarse structure.
Therefore, the proposed methodology is effective in reducing the elapsed time for dealing with very deep neural networks.
Numerical results confirm that the parareal neural network provides similar or better performance to the original network even with less training time.

The rest of this paper is organized as follows.
In~\Cref{Sec:Para}, we briefly summarize the parareal algorithm for time-dependent differential equations.
An abstract framework for the construction of the parareal neural network is introduced in~\Cref{Sec:PNN}.
In~\Cref{Sec:App}, we present how to apply the proposed methodology to two popular neural networks VGG-16~\cite{simonyan2014very} and ResNet-$1001$~\cite{he2016identity} with details.
Also, accelerated and accuracy-preserved results of parareal neural networks for VGG-16 and ResNet-1001 with datasets CIFAR-$10$, CIFAR-$100$~\cite{krizhevsky2009cifar}, MNIST~\cite{lecun1998gradient}, SVHN~\cite{netzer2011reading}, and ImageNet~\cite{deng2009imagenet} are given.
We conclude this paper with remarks in~\Cref{Sec:Conclu}.

%%%%%%%%%%%%%%%%%%%%%%%%%%%%%%%%%%%%%%%%%%%%%%%%%%%%%%%%%%%%%%%%%%%%%%
\section{The parareal algorithm}
\label{Sec:Para}
The parareal algorithm proposed by Lions et al.~\cite{lions2001resolution} is a parallel-in-time algorithm to solve time-dependent differential equations.
For the purpose of description, the following system of ordinary differential equations is considered:
\begin{equation}
       \dot{\bu}(t) = A \bu (t) \ \ \textrm{in} \ [0,T], \ \ \bu (0)=\bu_0,
       \label{eq_ode1}
       \end{equation}
where $A$:~$\mathbb{R}^m \rightarrow \mathbb{R}^m$ is an operator, $T>0$, and $\bu_0 \in \mathbb{R}^m$.
The time interval $[0,T]$ is decomposed into $N$ subintervals $0=T_0<T_1< \dots <T_N=T$.
First, an approximated solution $\{\bU_j^1\}_{j=0}^N$ of~\cref{eq_ode1} on the \textit{coarse} grid $\{ T_j \}_{j=0}^N$ is obtained by the backward Euler method with the step size $\Delta T_j= T_{j+1} - T_j$:
\begin{equation*}
       \frac{\bU_{j+1}^1- \bU_j^1}{\Delta T_j} = A \bU_{j+1}^1 , \ \ \bU_0^1 = \bu_0  \ \ \textrm{for} \ j=0,\dots ,N-1.
       \end{equation*}
Then in each time subinterval $[T_j,T_{j+1}]$, we construct a local solution $\bu_j^1$ by solving the following initial value problem:
\begin{equation}
       \dot{\bu}_j^1 (t) = A\bu_j^1 (t)  \ \ \textrm{in} \ [T_j,T_{j+1}], \ \ \bu_j^1(T_j)=\bU_j^1.
       \label{eq_ode2}
       \end{equation}
The computed solution $\bu_j^1$ does not agree with the exact solution $\bu$ in general since $\bU_j^1$ differs from $\bu(T_j)$.
For $k \geq 1$, a better coarse approximation $\{ \bU_j^{k+1} \}_{j=0}^N$ than $ \{ \bU_j^k \}_{j=0}^N$ is obtained by the \textit{coarse grid correction}: Let $\bU_0^{k+1}=\bU_0^k, \bS_0^k=0$; for $j=0, \dots, N-1$, we repeat the followings:
\begin{enumerate}
\item Compute the difference at the coarse node: $\bS_{j+1}^k = \bu_{j}^k(T_{j+1})-\bU_{j+1}^k$.
\item Propagate the difference to the next coarse node by the backward Euler method: \\$\frac{\bdelta_{j+1}^k-\bdelta_j^k} {\Delta T_j} = A \bdelta_{j+1}^k + \bS_j^k, \ \bdelta_0^k=0$.
\item Set $\bU_{j+1}^{k+1}=\bU_{j+1}^k+\bdelta_{j+1}^k$.
\end{enumerate}
That is, $\{ \bU_j^{k+1} \}_{j=0}^N$ is made by the correction with the propagated residual $\{ \bdelta_j^k \}_{j=0}^N$.
Using the updated coarse approximation $\{ \bU_j^{k+1} \}_{j=0}^N$, one obtains a new local solution $\bu_j^{k+1}$ in the same manner as~\cref{eq_ode2}:
\begin{equation}
       \dot{\bu}_j^{k+1}(t) = A \bu_j^{k+1} (t) \ \ \textrm{in} \ [T_j,T_{j+1}], \ \ \bu_j^{k+1}(T_j)=\bU_j^{k+1}.
       \label{eq_ode3}
       \end{equation}
It is well-known that $\bu_j^k$ converges to the exact solution $\bu$ uniformly as $k$ increases~\cite{bal2005convergence,gander2008nonlinear}.

\begin{figure}
\centering
\includegraphics[width=0.8\linewidth]{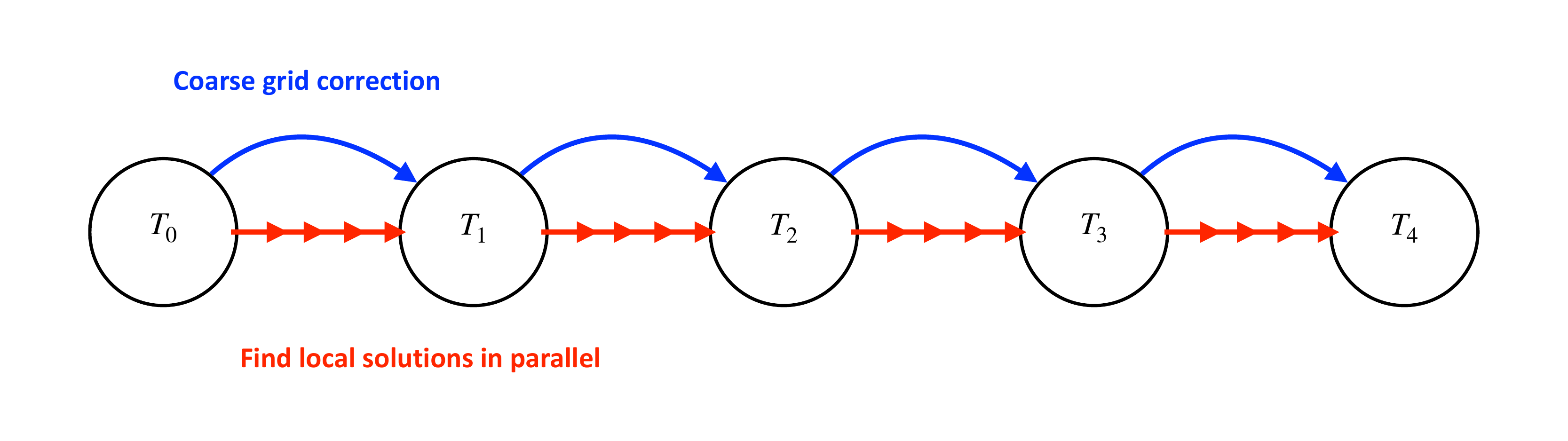}
\caption{Fine and coarse propagations in the parareal algorithm: the red lines which propagate from $T_j$ to $T_{j+1}$ represent~\cref{eq_ode2} and the blue curves which directly connect $T_j$ and $T_{j+1}$ represent~\cref{eq_ode3}.}
\label{parareal}
\end{figure}

Since~\cref{eq_ode3} can be solved independently in each time subinterval, we may assign the problem in each $[T_j,T_{j+1}]$ to the processor one by one and compute $\bu_j^{k+1}$ in parallel.
In this sense, the parareal algorithm is suitable for parallel computation on distributed memory architecture.
A diagram illustrating the parareal algorithm is presented in~\Cref{parareal}.

%%%%%%%%%%%%%%%%%%%%%%%%%%%%%%%%%%%%%%%%%%%%%%%%%%%%%%%%%%%%%%%%%%%%%%
\section{Parareal neural networks}
\label{Sec:PNN}
In this section, we propose a methodology to design a \textit{parareal neural network} by emulating the parareal algorithm introduced in~\Cref{Sec:Para} from a given feed-forward neural network. The resulting parareal neural network has an intrinsic parallel structure and is suitable for parallel computation using multiple GPUs with distributed memory simultaneously.

\subsection{Parallelized forward propagation}

\begin{figure}
  \centering
  \includegraphics[width=0.8\linewidth]{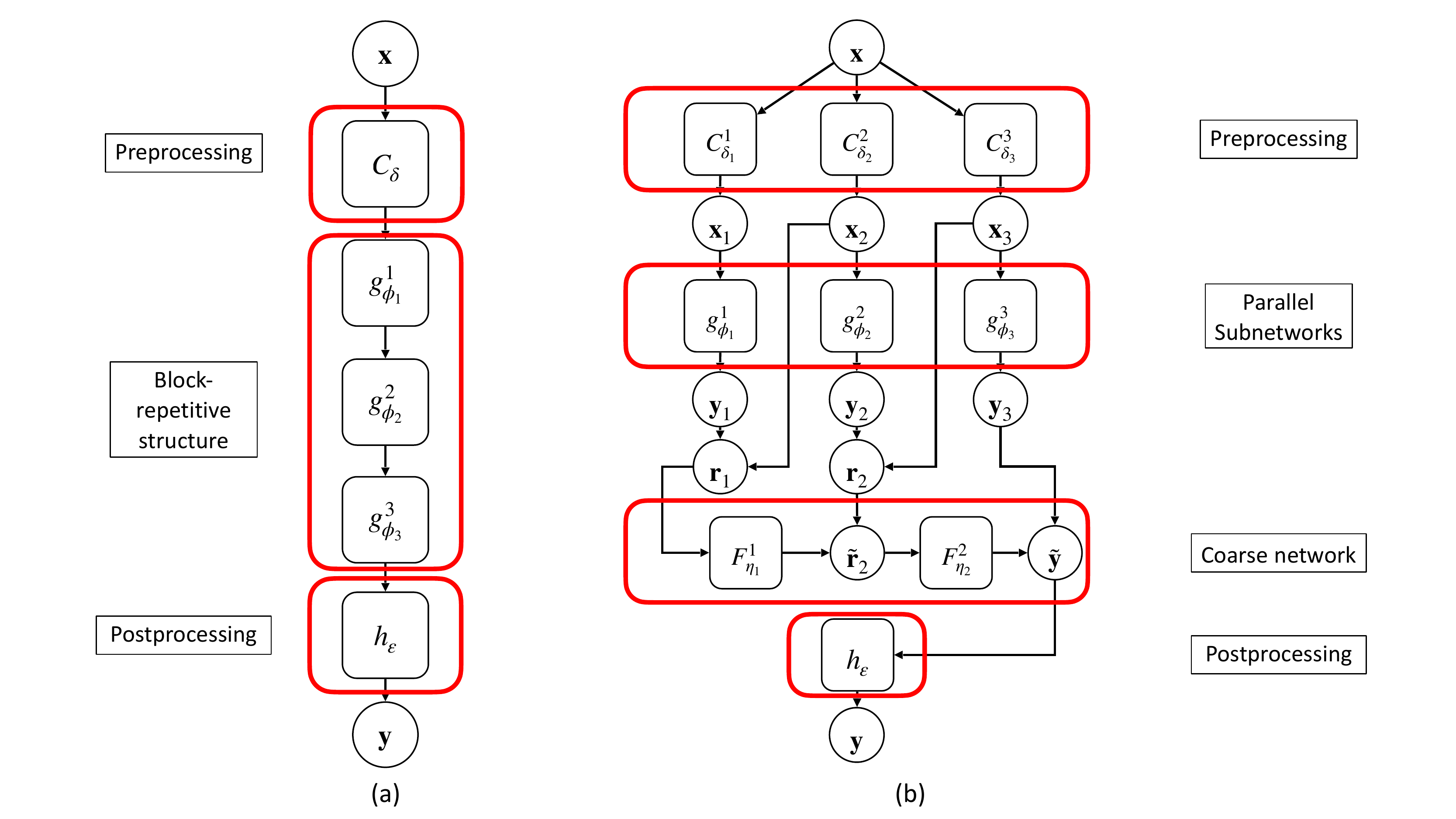}
  \caption{A feed-forward neural network and its corresponding parareal neural network: \textbf{(a)} Feed-forward neural network $f_{\theta}$, \textbf{(b)} Parareal neural network $\bar{f}_{\bar{\theta}}$ with $N$ parallel subnetworks~($N=3$).}
  \label{f2}
\end{figure}

Let $f_\theta$:~$X \rightarrow Y$ be a feed-forward neural network, where $X$ and $Y$ are the spaces of inputs and outputs, respectively, and $\theta$ is a vector consisting of parameters.
Since many modern neural networks such as~\cite{he2016identity,simonyan2014very,zagoruyko2016wide} have block-repetitive substructures, we may assume that $f_\theta$ can be written as the composition of three functions $C_{\delta}$:~$X \rightarrow W_0$, $g_{\phi}$:~$W_0 \rightarrow W_1$, and $h_{\varepsilon}$:~$W_1 \rightarrow Y$, i.e.,
\begin{equation*}
f_{\theta} = h_{\varepsilon} \circ g_{\phi} \circ C_{\delta}, \quad \theta = \delta \oplus \phi \oplus \varepsilon,
\end{equation*}
where $W_0$ and $W_1$ are vector spaces, $g_{\phi}$ is a block-repetitive substructure of $f_{\theta}$ with parameters $\phi$, $C_{\delta}$ is a \textit{preprocessing operator} with parameters $\delta$, and $h_{\varepsilon}$ is a \textit{postprocessing operator} with parameters $\varepsilon$.
Note that $\oplus$ represents a concatenation.
Examples of VGG-16~\cite{simonyan2014very} and ResNet-1001~\cite{he2016identity} will be given in~\Cref{Sec:App}.

For appropriate vector spaces $X_0$, $X_1$, \dots, $X_N$, we further assume that $g_{\phi}$ can be partitioned into $N$ subnetworks $\{ g_{\phi_j}^j$:~$X_{j-1} \rightarrow X_j \}_{j=1}^N$ which satisfy the followings:
\begin{itemize}
\item $X_0 = W_0$ and $X_N = W_1$,
\item $\phi = \bigoplus_{j=1}^N \phi_j$,
\item $g_{\phi} = g_{\phi_N}^N \circ g_{\phi_{N-1}}^{N-1} \circ \dots \circ g_{\phi_1}^1$.
\end{itemize}
See~\Cref{f2}(a) for a graphical description for the case $N=3$.
In the computation of $g_{\phi}$, the subnetworks $\{ g_{\phi_j}^{j} \}_{j=1}^N$ are computed in the sequential manner.
Regarding the subnetworks as subintervals of a time-dependent problem and adopting the idea of the parareal algorithm introduced in~\Cref{Sec:Para}, we construct a new neural network $\bar{f}_{\bar{\theta}}$:~$X \rightarrow Y$ which contains $\{ g_{\phi_j}^{j} \}_{j=1}^N$ as parallel subnetworks; the precise definition for parameters $\bar{\theta}$ will be given in~\cref{barf}.

Since the dimensions of the spaces $\{X_{j}\}_{j=0}^{N-1}$ are different for each $j$ in general,
we introduce preprocessing operators $C_{\delta_{j}}^j$:~$X \rightarrow X_{j-1}$ such that $C_{\delta_1}^1 = C_{\delta}$ and $C_{\delta_j}^j$ for $j = 2,\dots, N$ play similar roles to $C_{\delta}$; particular examples will be given in~\Cref{Sec:App}.
We write $\bx_j \in X_{j-1}$ and $\by_j \in X_j$ as follows:
\begin{equation}
\label{xjyj}
\bx_j = C_{\delta_j}^j (\bx) \textrm{ for } \bx \in X, \quad \by_j = g_{\phi_j}^j (\bx_j).
\end{equation}

Then, we consider neural networks $F_{\eta_{j}}^j$:~$X_j \rightarrow X_{j+1}$ with parameters $\eta_j$ for $j \ge 1$ such that it approximates $g_{\phi_{j+1}}^{j+1}$ well while it has a cheaper computational cost than $g_{\phi_{j+1}}^{j+1}$, i.e., $F_{\eta_j}^j \approx g_{\phi_{j+1}}^{j+1}$ and $\dim (\eta_j) \ll \dim (\phi_{j+1})$.
Emulating the coarse grid correction of the parareal algorithm, we assemble a network called \textit{coarse network} with building blocks $F_{\eta_j}^j$.
With inputs $\bx_{j+1}$, $\by_j$, and an output $\by \in Y$, the coarse network is described as follows:
\begin{subequations}
\label{correct}
\begin{align}
  \label{rk}
  \br_N &= \mathbf{0}, \quad \br_j = \by_j - \bx_{j+1} \quad \textrm{ for } j=1,\dots,N-1,  \\
  \label{tilderk}
  \tilde{\br}_1 &= \br_1, \quad \tilde{\br}_{j+1} = \br_{j+1} + F_{\eta_j}^j (\tilde{\br}_j) \quad \textrm{ for } j=1,\dots,N-1, \\
  \label{tildey}
  \tilde{\by} &= \by_N + \tilde{\br}_N.
\end{align}
\end{subequations}
That is, in the coarse network, the residual $\br_j$  at the interface between layers $g_{\phi_j}^j$ and $g_{\phi_{j+1}}^{j+1}$ propagates through shallow neural networks $F_{\eta_1}^1$, \dots, $F_{\eta_{N-1}}^{N-1}$.
Then the propagated residual is added to the output.

% Algorithm: Forward propagation
\begin{algorithm}[]
\caption{Forward propagation of the parareal neural network $\bar{f}_{\bar{\theta}}$}
\begin{algorithmic}[]
\label{Alg:para2}
\STATE Broadcast {$\mathbf{x}$} to all processors.
\FOR{$j=1,2,\dots,N$ \textbf{in parallel}}
\item \vspace{-0.5cm}\begin{align*}
  \displaystyle \mathbf{x}_{j} = C^{j}_{\delta_{j}}(\mathbf{x}), \ \mathbf{y}_{j} = g^{j}_{\phi_{j}}(\mathbf{x}_{j})
\end{align*}\vspace{-0.4cm}
\ENDFOR
\STATE Gather {$\mathbf{x}_{j}, \mathbf{y}_{j}$} from all processors.
\FOR{$j=1,2,\dots,N-1$}
\item \vspace{-0.5cm}\begin{equation*}
  \displaystyle \mathbf{r}_{j} = \mathbf{y}_{j} - \mathbf{x}_{j+1}
\end{equation*}\vspace{-0.4cm}
\ENDFOR
\STATE $\br_N = 0$, $\tilde{\br}_1 = \br_1$.
\FOR{$j=1,2,\dots,N-1$}
\item \vspace{-0.5cm}\begin{equation*}
  \displaystyle \tilde{\br}_{j+1} = \mathbf{r}_{j+1} + F^{j}_{\eta_{j}}(\tilde{\br}_{j})
\end{equation*}\vspace{-0.4cm}
\ENDFOR
\STATE $\displaystyle \tilde{\mathbf{y}} = \mathbf{y}_N + \tilde{\br}_N$
\STATE $\displaystyle \mathbf{y} = h_\varepsilon(\tilde{\by} )$
\end{algorithmic}
\end{algorithm}

% Algorithm: Backward propagation
\begin{algorithm}[]
\caption{Gradient computation for the parareal neural network $\bar{f}_{\bar{\theta}}$}
\begin{algorithmic}[]
\label{Alg:para3}
\STATE $\displaystyle \frac{\partial \bar{f}_{\bar{\theta}}}{\partial \varepsilon} = \frac{\partial h_{\varepsilon}}{\partial \varepsilon}$, $\displaystyle \mathcal{D}_{N}=\frac{\partial h_{\varepsilon}}{\partial \tilde{\by}}$, $\displaystyle \mathcal{D}_0=0$
\FOR{$j=N-1, N-2, \dots 1$}
\item \vspace{-0.5cm}\begin{align*}
\displaystyle \frac{\partial \bar{f}_{\bar{\theta}}}{\partial \eta_j} = \mathcal{D}_{j+1} \cdot \frac{\partial F_{\eta_j}^j}{\partial \eta_j}, \ \displaystyle \mathcal{D}_{j} = \mathcal{D}_{j+1} \cdot \frac{\partial F_{\eta_{j}}^{j}}{\partial \tilde{\br}_{j}}
\end{align*}\vspace{-0.4cm}
\ENDFOR
\STATE Send $\mathcal{D}_{j-1}$ and $\mathcal{D}_{j}$  to the $j$-th processor.\\
\FOR{$j=1,2,\dots,N$ \textbf{in parallel}}
\item \vspace{-0.5cm}\begin{align*}
\displaystyle \frac{\partial \bar{f}_{\bar{\theta}}}{\partial \phi_j} = \mathcal{D}_{j} \cdot \frac{\partial g_{\phi_j}^j}{\partial \phi_j}, \ \displaystyle \frac{\partial \bar{f}_{\bar{\theta}}}{\partial \delta_j} = \left( \mathcal{D}_{j} \cdot \frac{\partial g_{\phi_j}^j}{\partial \bx_j} - \mathcal{D}_{j-1} \right) \frac{\partial C_{\delta_j}^j}{\partial \delta_j}
\end{align*}\vspace{-0.4cm}
\ENDFOR
\end{algorithmic}
\end{algorithm}

Finally, the parareal neural network $\bar{f}_{\bar{\theta}}$ corresponding to the original network $f_{\theta}$ is defined as
\begin{equation}
\label{barf}
\bar{f}_{\bar{\theta}} (\bx) = h_\varepsilon(\tilde{\by} ), \quad
\bar{\theta} = \left( \bigoplus_{j=1}^N (\delta_j \oplus \phi_j ) \right)
\oplus \left( \bigoplus_{j=1}^{N-1} \eta_j \right) \oplus \varepsilon.
\end{equation}
That is, $\bar{f}_{\bar{\theta}}$ is composed of the preprocessing operators $\{C_{\delta_j}^j\}$, parallel subnetworks $\{g_{\phi_j}^j\}$, the coarse network $\{F_{\eta_j}^j\}$, and the postprocessing operator $h_\varepsilon$.
\Cref{f2}(b) illustrates $\bar{f}_{\bar{\theta}}$.
Since each $g_{\phi_j}^j \circ C_{\delta_j}^j$ lies in parallel, all computations related to $g_{\phi_j}^j \circ C_{\delta_j}^j$ can be done independently; parallel structures of forward and backward propagations for $\bar{f}_{\bar{\theta}}$ are described in~\Cref{Alg:para2} and~\Cref{Alg:para3}, respectively; detailed derivation of~\Cref{Alg:para3} will be provided in~\Cref{Subsec:Back}.
Therefore, multiple GPUs can be utilized to process $\{g_{\phi_j}^j \circ C_{\delta_j}^j\}$ simultaneously for each $j$.
In this case, one may expect significant decrease of the elapsed time for training $\bar{f}_{\bar{\theta}}$ compared to the original network $f_{\theta}$.
On the other hand, the coarse network cannot be parallelized since $\{F_{\eta_j}^j\}$ is computed in the sequential manner.
One should choose $F_{\eta_j}^j$ whose computational cost is as cheap as possible in order to reduce the bottleneck effect of the coarse network.

In the following proposition, we show that the proposed parareal neural network $\bar{f}_{\bar{\theta}}$ is constructed consistently in the sense that it recovers the original neural network $f_{\theta}$ under a simplified setting.

% Proposition: Equivalence between the original and parareal neural networks
\begin{proposition}[Consistency]
\label{Prop:equiv}
Assume that the original network $f_{\theta}$ is linear and $F_{\eta_j}^j = g_{\phi_{j+1}}^{j+1}$ for $j= 1, \dots, N-1$.
Then we have $\bar{f}_{\bar{\theta}} (\bx) = f_{\theta} (\bx)$ for all $\bx \in X$.
\end{proposition}
\begin{proof}
  We define a function $P_j$:~$X \rightarrow X_{j+1}$ inductively as follows:
  \begin{equation}
  \label{P_recur}
  P_0 (\bx) = 0, \quad
  P_j(\bx) = F_{\eta_j}^j \left( (g_{\phi_j}^j \circ C_{\delta_j}^j ) (\bx) - C_{\delta_{j+1}}^{j+1} (\bx) + P_{j-1} (\bx) \right), \quad 1 \leq j \leq N-1.
  \end{equation}
  Then it follows that
  \begin{equation}
  \label{f_P}
  \bar{f}_{\bar{\theta}} (\bx) = h_{\varepsilon} \left((g_{\phi_N}^N \circ C_{\delta_N}^N) (\mathbf{x}) + P_{N-1}(\bx) \right).
  \end{equation}

  First, we show by mathematical induction that
  \begin{equation}
  \label{P}
  P_j (\bx) = (g_{\phi_{j+1}}^{j+1} \circ g_{\phi_{j}}^{j} \circ \dots g_{\phi_{1}}^{1} \circ C_{\delta_1}^1) (\bx) - (g_{\phi_{j+1}}^{j+1} \circ C_{\delta_{j+1}}^{j+1}) (\bx), \quad 1 \leq j \leq N-1.
  \end{equation}
  The case $j=1$ is straightforward from~\cref{P_recur}.
  Suppose that~\cref{P} holds for $j = m-1$.
  Since the original network $f_{\theta}$ is linear and $F_{\eta_j}^j = g_{\phi_{j+1}}^{j+1}$, we get
  \begin{equation*} \begin{split}
  P_m(\bx) &= F_{\eta_m}^m \left( (g_{\phi_m}^m \circ C_{\delta_m}^m ) (\bx) - C_{\delta_{m+1}}^{m+1} (\bx) + P_{m-1} (\bx) \right) \\
  &= (g_{\phi_{m+1}}^{m+1} \circ g_{\phi_{m}}^{m} \circ C_{\delta_m}^m) (\bx) - (g_{\phi_{m+1}}^{m+1} \circ C_{\delta_{m+1}}^{m+1}) (\bx) + (g_{\phi_{m+1}}^{m+1} \circ P_{m-1}) (\bx) \\
  &= (g_{\phi_{m+1}}^{m+1} \circ g_{\phi_{m}}^{m} \circ \dots g_{\phi_{1}}^{1} \circ C_{\delta_1}^1) (\bx) - (g_{\phi_{m+1}}^{m+1} \circ C_{\delta_{m+1}}^{m+1}) (\bx),
  \end{split} \end{equation*}
  where the last equality is due to the induction hypothesis.
  Hence,~\cref{P} also holds for $j=m$, which implies that it is true for all $j$.

  Combining~\cref{f_P} and~\cref{P}, we obtain
  \begin{equation*} \begin{split}
  \bar{f}_{\bar{\theta}} (\bx) &= h_{\varepsilon} \left((g_{\phi_N}^N \circ C_{\delta_N}^N) (\mathbf{x}) + P_{N-1}(\bx) \right) \\
  &= (h_{\varepsilon} \circ g_{\phi_{N}}^{N} \circ g_{\phi_{N-1}}^{N-1} \circ \dots g_{\phi_{1}}^{1} \circ C_{\delta_1}^1 ) (\bx) \\
  &= (h_{\varepsilon} \circ g_{\phi_{N}}^{N} \circ g_{\phi_{N-1}}^{N-1} \circ \dots g_{\phi_{1}}^{1} \circ C_{\delta} ) (\bx) \\
  &= f_{\theta} (\bx),
  \end{split} \end{equation*}
  which completes the proof.
\end{proof}

\Cref{Prop:equiv} presents a guideline on how to design the coarse network of $\bar{f}_{\bar{\theta}}$.
Under the assumption that $f_{\theta}$ is linear, a sufficient condition to ensure that $\bar{f}_{\bar{\theta}} = f_{\theta}$ is $F_{\eta_j}^j = g_{\phi_{j+1}}^{j+1}$ for all $j$.
Therefore, we can say that it is essential to design the coarse network with $F_{\eta_j}^j \approx g_{\phi_{j+1}}^{j+1}$ to ensure that the performance of $\bar{f}_{\bar{\theta}}$ is as good as that of $f_{\theta}$.
Detailed examples will be given in~\Cref{Sec:App}.

On the other hands, the propagation of the coarse network in the parareal neural network is similar to gradient boosting~\cite{friedman2001greedy,mason2000boosting}, one of the ensemble techniques.
In \cref{correct}, the coarse network satisfies $F_{\eta_j}^{j}(\tilde{\br}_j)=\tilde{\br}_{j+1}-\br_{j+1}$.
Let $\mathcal{L}(\bx,\by)=\frac{1}{2}\| \bx-\by \|^2$, then we have
\begin{equation*}
  \displaystyle F_{\eta_j}^{j}(\tilde{\br}_j)=\tilde{\br}_{j+1}-\br_{j+1}=- \frac{\partial \mathcal{L}(\tilde{\br}_{j+1},\br_{j+1})}{\partial \br_{j+1}}.
\end{equation*}
From the viewpoint of gradient boosting, the propagation of coarse network can be expressed by the following gradient descent method:
\begin{equation*}
  \displaystyle \tilde{\br}_{j+1}= \br_{j+1} - \frac{\partial \mathcal{L}(\tilde{\br}_{j+1},\br_{j+1})}{\partial \br_{j+1}}.
\end{equation*}
% In other words, as an optimization problem, it is the same as finding $\eta_{j}$ such that
% \begin{equation*}
%   \displaystyle \eta_{j}=\arg\min_{\eta} \mathcal{L}(\tilde{\br}_{j+1}, \br_{j+1}+F_{\eta}^{j}(\tilde{\br}_j)).
% \end{equation*}
Thus, it can be understood that the coarse network constructed by emulating the coarse grid correction of the parareal algorithm corrects the residuals at the interface by the gradient descent method, i.e., it reduces the difference between $\bx_{j+1}$ and $\by_{j}$ at each interface.

Furthermore, we can think of the preprocessing $C_{\delta_{j}}^{j}$ and subnetwork $g_{\phi_{j}}^{j}$ in parareal neural network as a single shallow network $g_{\phi_{j}}^{j} \circ C_{\delta_{j}}^{j}$.
Then, the parareal neural network can be thought of as a network in which several shallow neural networks are stacked, such as the Stacked generalization~\cite{wolpert1992stacked} of the ensemble technique.
Thanks to the coarse network, the parareal neural network does not simply stack the shallow networks, but behaves like a deep neural network that sequentially computes the parallel subnetworks as mentioned in~\Cref{Prop:equiv}.

\subsection{Details of backward propagation}
\label{Subsec:Back}
We present a detailed description on the backward propagation for the parareal neural network $\by = \bar{f}_{\bar{\theta}}(\bx)$.
Partial derivatives $\frac{\partial \bar{f}_{\bar{\theta}}}{\partial \varepsilon}$ and $\frac{\partial \by}{\partial \tilde{\by}}$ regarding to the postprocessing operator $h_{\varepsilon}$ are computed directly from~\cref{barf}:
\begin{equation*}
\frac{\partial \bar{f}_{\bar{\theta}}}{\partial \varepsilon} = \frac{\partial h_{\varepsilon}}{\partial \varepsilon}, \quad
\frac{\partial \by}{\partial \tilde{\by}} = \frac{\partial h_{\varepsilon}}{\partial \tilde{\by}}.
\end{equation*}
It is clear from~\cref{correct} that
\begin{equation}
\label{ingredient1}
\frac{\partial \br_j}{\partial \by_j} = 1, \quad
\frac{\partial \br_j}{\partial \bx_{j+1}} = -1, \quad
\frac{\partial \tilde{\br}_j}{\partial \br_j} = 1, \quad
\frac{\partial \tilde{\by}}{\partial \by_N} = 1, \quad
\frac{\partial \tilde{\by}}{\partial \tilde{\br}_N} = 1.
\end{equation}
Moreover, by~\cref{tilderk}, we get
\begin{equation}
\label{ingredient2}
\frac{\partial \tilde{\br}_{j+1}}{\partial \eta_j} = \frac{\partial F_{\eta_j}^j}{\partial \eta_j}, \quad
\frac{\partial \tilde{\br}_{j+1}}{\partial \tilde{\br}_j} = \frac{\partial F_{\eta_j}^j}{\partial \tilde{\br}_j}.
\end{equation}
Invoking the chain rule with~\cref{ingredient1} and~\cref{ingredient2}, $\frac{\partial \bar{f}_{\bar{\theta}}}{\partial \eta_j}$ is described as
\begin{equation}
\label{etak}
\frac{\partial \bar{f}_{\bar{\theta}}}{\partial \eta_j}
= \frac{\partial \by}{\partial \tilde{\by}} \displaystyle \frac{\partial \tilde{\by}}{\partial \tilde{\br}_N} \left( \prod_{l=j+1}^{N-1} \frac{\partial \tilde{\br}_{l+1}}{\partial \tilde{\br}_{l}} \right) \frac{\partial \tilde{\br}_{j+1}}{\partial \eta_j}
= \frac{\partial h_{\varepsilon}}{\partial \tilde{\by}} \left( \prod_{l=j+1}^{N-1} \frac{\partial F_{\eta_l}^l}{\partial \tilde{\br}_{l}} \right) \frac{\partial F_{\eta_j}^j}{\partial \eta_j}.
 \end{equation}

On the other hand, partial derivatives $\frac{\partial \bx_j}{\partial \delta_j} $, $\frac{\partial \by_j}{\partial \phi_j}$, and $\frac{\partial \by_j}{\partial \bx_j}$ can be computed in parallel by~\cref{xjyj}:
\begin{equation}
\label{ingredient3}
\frac{\partial \bx_j}{\partial \delta_j} = \frac{\partial C_{\delta_j}^j}{\partial \delta_j}, \quad
\frac{\partial \by_j}{\partial \phi_j} = \frac{\partial g_{\phi_j}^j}{ \partial \phi_j}, \quad
\frac{\partial \by_j}{\partial \bx_j} = \frac{\partial g_{\phi_j}^j}{ \partial \bx_j} .
\end{equation}
Using~\cref{ingredient1},~\cref{ingredient2}, and~\cref{ingredient3}, it follows that
\begin{align}
\label{phij}
\frac{\partial \bar{f}_{\bar{\theta}}}{\partial \phi_j} &= \begin{cases}\displaystyle \frac{\partial \by}{\partial \tilde{\by}} \displaystyle \frac{\partial \tilde{\by}}{\partial \tilde{\br}_N} \left( \prod_{l=j}^{N-1} \frac{\partial \tilde{\br}_{l+1}}{\partial \tilde{\br}_{l}} \right) \frac{\partial \tilde{\br}_j}{\partial \br_j} \frac{\partial \br_{j}}{\partial \by_j} \frac{\partial \by_j}{\partial \phi_j} & \textrm{ if } j < N  , \notag\\
\displaystyle \frac{\partial \by}{\partial \tilde{\by}} \frac{\partial \tilde{\by}}{\partial \by_N} \frac{\partial \by_N}{\partial \phi_N} & \textrm{ if } j = N \end{cases} \notag \\
&= \frac{\partial h_{\varepsilon}}{\partial \tilde{\by}} \left( \prod_{l=j}^{N-1} \frac{\partial F_{\eta_l}^l}{\partial \tilde{\br}_{l}} \right) \frac{\partial g_{\phi_j}^j}{\partial \phi_j}.
\end{align}
Similarly, with the convention $\frac{\partial F_{\eta_0}^0}{\partial \tilde{\br}_0} = 0$, we have
\begin{align}
\label{deltaj}
\frac{\partial \bar{f}_{\bar{\theta}}}{\partial \delta_j} &=
\begin{cases}\displaystyle \frac{\partial \by}{\partial \tilde{\by}} \displaystyle \frac{\partial \tilde{\by}}{\partial \tilde{\br}_N} \left( \prod_{l=j}^{N-1} \frac{\partial \tilde{\br}_{l+1}}{\partial \tilde{\br}_{l}} \right) \left( \frac{\partial \tilde{\br}_j}{\partial \br_j} \frac{\partial \br_{j}}{\partial \by_j} \frac{\partial \by_j}{\partial \bx_j} + \frac{\partial \tilde{\br}_j}{\partial \tilde{\br}_{j-1}} \frac{\partial \tilde{\br}_{j-1}}{\partial \br_{j-1}} \frac{\partial \br_{j-1}}{\partial \bx_j} \right) \frac{\partial \bx_j}{\partial \delta_j} & \textrm{ if } j < N  , \notag\\
\displaystyle \frac{\partial \by}{\partial \tilde{\by}} \left( \frac{\partial \tilde{\by}}{\partial \by_N} \frac{\partial \by_N}{\partial \bx_N} + \frac{\partial \tilde{\by}}{\partial \tilde{\br}_N} \frac{\partial \tilde{\br}_N}{\partial \tilde{\br}_{N-1}} \frac{\partial \tilde{\br}_{N-1}}{\partial \br_{N-1}} \frac{\partial \br_{N-1}}{\partial \bx_N} \right) \frac{\partial \bx_N}{\partial \delta_N} & \textrm{ if } j = N \end{cases} \notag \\
&= \frac{\partial h_{\varepsilon}}{\partial \tilde{\by}} \left( \prod_{l=j}^{N-1} \frac{\partial F_{\eta_l}^l}{\partial \tilde{\br}_{l}} \right) \left( \frac{\partial g_{\phi_j}^j}{\partial \bx_j} - \frac{\partial F_{\eta_{j-1}}^{j-1}}{\partial \tilde{\br}_{j-1}} \right) \frac{\partial C_{\delta_j}^j}{\partial \delta_j}.
\end{align}
For efficient computation, the value of $\frac{\partial h_{\varepsilon}}{\partial \tilde{\by}} \left( \prod_{l=j}^{N-1} \frac{\partial F_{\eta_l}^l}{\partial \tilde{\br}_l} \right)$ can be stored during the evaluation of~\cref{etak} and then used in~\cref{phij} and~\cref{deltaj}.
Such a technique is described in~\Cref{Alg:para3}.

%%%%%%%%%%%%%%%%%%%%%%%%%%%%%%%%%%%%%%%%%%%%%%%%%%%%%%%%%%%%%%%%%%%%%%
\section{Applications}
\label{Sec:App}
In this section, we present applications of the proposed methodology to two existing convolutional neural networks~(CNNs): VGG-16~\cite{simonyan2014very} and ResNet-1001~\cite{he2016identity}.
For each network, we deal with details on the construction of parallel subnetworks and a coarse network.
Also, numerical results are presented showing that the proposed parareal neural network gives comparable or better results than given feed-forward neural network and other variants in terms of both training time and classification performance.

First, we present details on the datasets we used.
The CIFAR-$m$~($m=10, 100$) dataset consists of $32 \times 32$ colored natural images and includes 50,000 training and 10,000 test samples with $m$ classes.
The SVHN dataset is composed of $32 \times 32$ colored digit images; there are 73,257 and 26,032 samples for training and test, respectively, with additional 531,131 training samples.
However, we did not use the additional ones for training.
MNIST is a classic dataset which contains handwritten digits encoded in $28 \times 28$ grayscale images.
It includes 55,000 training, 5,000 validation, and 10,000 test samples.
In our experiments, the training and validation samples are used as training data and the test samples as test data.
ImageNet is a dataset which contains 1000 classes of $224 \times 224$ colored natural images.
It includes 1,280,000 training and 50,000 test samples.

We adopted a data augmentation technique in~\cite{lee2015deeply} for CIFAR datasets; four pixels are padded on each side of images, and $32 \times 32$ crops are randomly sampled from the padded images and their horizontal flips.
All neural networks in this section were trained using the stochastic gradient descent with the batch size $128$, weight decay $0.0005$, Nesterov momentum $0.9$, and weights initialized as in~\cite{he2015delving}.
The initial learning rate was set to $0.1$, and was reduced by a factor of $10$ in the $80$th and $120$th epochs.
For ImageNet datasets, the input image is $224 \times 224$ randomly cropped from a resized image using the scale and aspect ratio augmentation~\cite{szegedy2015going}.
Hyperparameter settings are the same as other cases except the followings; the weight decay $0.0001$, total epoch $90$, and the learning rate was reduced by a factor of 10 in the 30th and 60th epochs.
All networks were implemented in Python with PyTorch and
all computations were performed on a cluster equipped with Intel Xeon Gold 5515 (2.4GHz, 20C) CPUs, NVIDIA Titan RTX GPUs, and the operating system Ubuntu 18.04 64bit.

% Subsection: VGG-16
\subsection{VGG-16}
In general, CNNs without skip connections~(see, e.g.,~\cite{krizhevsky2012imagenet,simonyan2014very}) can be represented as
\begin{equation*}
 \mathbf{x}_l = H_l(\mathbf{x}_{l-1}),
 \end{equation*}
where $\mathbf{x}_l$ is an output of the $l$th layer of the network and $H_l$ is a nonlinear transformation consisting of convolutions, batch normalization and ReLU activation.
Each layer of VGG-16~\cite{simonyan2014very}, one of the most popular CNNs without skip connections, consists of multiple $3 \times 3$ convolutions.
The network consists of 5 stages of convolutional blocks and 3 fully connected layers. Each convolutional block is a composition of double or triple convolutions and a max pooling operation.

\begin{table}
  \caption{Architecture of VGG-16 for the dataset ImageNet.
  A layer consisting of an $n \times n$ convolution with $k$-channel ouput, a max pooling with kernel size $2$ and stride $2$, and a $k$-way fully connected layer are denoted by $[n \times n, k]$, maxpool, and $k$-d~FC, respectively.}
  \label{vgg}
  \centering
  \begin{tabular}{c|c|c}
    \toprule
    Layer                                  & Output size & VGG-16                                                                    \\
    \midrule
    Preprocessing                               & $224 \times 224$   & \begin{tabular}[l]{@{}l@{}}{[}$3 \times 3, 64${]}\end{tabular}                                                                   \\ \hline
    \multirow{5}{*}{\begin{tabular}[c]{@{}c@{}}Block-repetitive\\ substructure\end{tabular}}& $112 \times 112$   & \begin{tabular}[l]{@{}l@{}}{[}$3 \times 3, 64${]} $+$ maxpool\end{tabular}       \\ \cline{2-3}
                                                & $56 \times 56$     & \begin{tabular}[c]{@{}l@{}}{[}$3 \times 3, 128${]}$\times 2$ $+$ maxpool\end{tabular}      \\ \cline{2-3}
                                                & $28 \times 28$     & \begin{tabular}[c]{@{}l@{}}{[}$3 \times 3, 256${]}$\times 3$ $+$ maxpool\end{tabular}      \\ \cline{2-3}
                                                & $14 \times 14$     & \begin{tabular}[c]{@{}l@{}}{[}$3 \times 3, 512${]}$\times 3$ $+$ maxpool\end{tabular}      \\ \cline{2-3}
                                                & $7 \times 7$       & \begin{tabular}[c]{@{}l@{}}{[}$3 \times 3, 512${]}$\times 3$ $+$ maxpool\end{tabular}      \\ \hline
    Postprocessing                              & $1 \times 1$       & \begin{tabular}[c]{@{}l@{}}{[}4096-d FC{]}$\times 2$ $+$ 1000-d FC\end{tabular} \\
    \bottomrule
    \end{tabular}
  \end{table}

\subsubsection{Parareal transformation}
First, we describe the structure of VGG-16 which was designed for the classification problem of ImageNet~\cite{deng2009imagenet} with the terminology introduced in~\Cref{Sec:PNN}.
Inputs for VGG-16 are 3-channel images with $224 \times 224$ pixels, i.e., $X = \mathbb{R}^{3 \times 224 \times 224}$.
The output space $Y$ is given by $Y = \mathbb{R}^{1000}$, where $1,000$ is the number of classes of ImageNet.
We set the preprocessing operator $C_{\delta}$:~$X \rightarrow W_0 = \mathbb{R}^{64 \times 224 \times 224}$ by the first $3 \times 3$ convolution layer in VGG-16.
We refer to the remaining parts of VGG-16 as the block-repetitive structure $g_{\phi}$:~$W_0 \rightarrow W_1$ with $W_1 = \mathbb{R}^{512 \times 7 \times 7}$ except for the last three fully connected layers.
Finally, the postprocessing operator $h_{\varepsilon}$:~$W_1 \rightarrow Y$  is the composition of the three fully connected layers.
\Cref{vgg} shows the detailed architecture of VGG-16.

In order to construct a parareal neural network with $N$ parallel subnetworks for VGG-16, we have to specify its components $g_{\phi_j}^j$, $C_{\delta_j}^j$, and $F_{\eta_j}^j$.
For simplicity, we assume that $N = 4$.
We decompose $g_{\phi}$ into $4$ parts such that the output size of each part is $56 \times 56$, $28 \times 28$, $14 \times 14$, and $7 \times 7$, respectively.
Then, the block-repetitive structure $g_{\phi}$ can be decomposed as
\begin{equation*}
g_{\phi} = g_{\phi_{4}}^{4} \circ g_{\phi_{3}}^{3} \circ g_{\phi_{2}}^{2} \circ g_{\phi_{1}}^{1},
\end{equation*}
where each of $g_{\phi_{j}}^{j}$:~$X_{j-1} \rightarrow X_{j}$ with
\begin{align*}
X_{j} = \begin{cases}\mathbb{R}^{128 \times 56 \times 56} & \textrm{ for } j=1\\ \mathbb{R}^{256 \times 28 \times 28} & \textrm{ for } j=2,\\ \mathbb{R}^{512 \times 14 \times 14} & \textrm{ for } j=3,\\ \mathbb{R}^{512 \times 7 \times 7} & \textrm{ for } j=4,\end{cases} \quad \phi = \bigoplus_{j=1}^4 \phi_j.
\end{align*}

Recall that the main role of the preprocessing operator $C_{\delta_j}^j$:~$X \rightarrow X_{j-1}$ is to transform an input $\bx \in X$ to fit in the space $X_{j-1}$.
In this perspective, we simply set $C_{\delta_1}^1 = C_{\delta}$ and $C_{\delta_j}^j$ for $j>1$ by a $1 \times 1$ convolution to match the number of channels after appropriate number of $3 \times 3$ max pooling layers with stride $2$ to match the image size.

According to~\Cref{Prop:equiv}, it is essential to design the coarse network such that $F_{\eta_j}^{j} \approx g_{\phi_{j+1}}^{j+1}$ in order to ensure the performance of the parareal neural network.
We simply define $F_{\eta_{j}}^j$:~$X_j \rightarrow X_{j+1}$ by the composition of two $3 \times 3$ convolutions and a max pooling with kernel size $2$ and stride $2$, which has a simplified structure of $g_{\phi_{j+1}}^{j+1}$ with fewer parameters.

\subsubsection{Numerical results}
\begin{table}
  \caption{Error rates~(\%) on the ImageNet dataset of VGG-16 and Parareal~VGG-16-4.}
  \label{vggresult}
  \centering
  \resizebox{0.9\linewidth}{!}{
  \begin{tabular}{ccccc}
    \toprule
    Network &\begin{tabular}{c}Subnetwork \\ Parameters \end{tabular}&\begin{tabular}{c}Coarse network \\ Parameters \end{tabular}& \begin{tabular}{c}Total \\ Parameters \end{tabular}& Error rate (\%) \\
    \midrule
    VGG-16 &-&-& 138.4M& 32.30\\
    \midrule
    Parareal VGG-16-4 &3.7M& 9.1M&147.5M& 30.97\\
    \bottomrule
  \end{tabular}}
\end{table}
We present the comparison results with Parareal VGG-16-4 and VGG-16 on ImageNet dataset.
Note that Parareal VGG-16-4 denotes the parareal neural network version of VGG-16 with $N=4$.
\Cref{vggresult} shows that the error rate of Parareal VGG-16-4 is smaller than that of VGG-16.
From this result, it can be seen that the accuracy is guaranteed even when the parareal neural network is applied to a network with small number of layers.

\begin{table}
  \caption{Forward/backward computation time for VGG-16 and Parareal VGG-16-4.
  It is a measure of the time taken in one iteration for ImageNet dataset input $\bx \in \mathbb{R}^{3 \times 224 \times 224}$ with batch size $128$.
  One interation means one step of updating all parameters with SGD.
  }
  \label{gradvgg}
  \centering
  \resizebox{\linewidth}{!}{
  \begin{tabular}{lccccc}
    \toprule
    \multicolumn{6}{c}{Virtual wall-clock time~(ms)}\\
    \cmidrule(r){1-6}
    Network & Preprocessing     & \begin{tabular}{c}Parallel\\subnetworks\end{tabular}    & \begin{tabular}{c}Coarse\\network\end{tabular} & Postprocessing  & Total \\
    \midrule
    VGG-16&18.73/321.22&187.21/3302.11&-&3.33/5.82&209.27/3629.15\\
    \midrule
    Parareal VGG-16-4&18.78/334.10&46.59/453.25&62.84/1413.49&2.96/5.90&131.17/2206.74\\
    % VGG-16&0.90/11.26&7.65/269.2&-&0.05/1.77&8.60/282.23\\
    % \midrule
    % Parareal VGG-16-4&0.96/15.02&1.81/43.37&1.44/139.21&0.04/2.72&4.25/200.52\\
    \bottomrule
  \end{tabular}}
\end{table}

Next, we investigate the elapsed time for forward and backward propagations of parareal neural networks, which are the most time-consuming part of network training.
\Cref{gradvgg} shows the virtual wall-clock time for forward and backward computation of VGG-16 and Parareal VGG-16-4 for the input $\bx \in \mathbb{R}^{3 \times 224 \times 224}$.
Note that the virtual wall-clock time is measured under the assumption that all parallelizable procedures indicated in~\Cref{Alg:para3} are executed simultaneously.
It means that, it excludes the communication time among GPUs.
In~\Cref{gradvgg}, even if the VGG-16 has the small number of layers, it can be seen that the computation time of Parareal-VGG-16-4 is reduced in terms of virtual wall clock time.

Now, we compare the performance of the proposed methodology to data parallelism.
In data parallelism, the batch is split into subsets at each epoch.
Each subset is assigned one for each GPU and the gradient corresponding to the subset is computed in parallel.
Then the parameters of the neural network are updated by the averaged gradient over all subsets.
In what follows, Data Parallel VGG-16-4 denotes the data parallelized VGG-16 with $4$ GPUs.
We compare the error rate and the wall-clock time of each parallelized network with the ImageNet dataset.
To investigate the speed-up provided by each parallelism, we use the relative speed-up~(RS) introduced in \cite{fok2018decoupling}, which is defined by
\begin{equation}
\label{RS}
\mathrm{RS} = \frac{t_r-t_p}{t_r},
\end{equation}
where $t_p$ is the total elapsed time taken to complete training of the given parallelism and $t_r$ is that of the given feed-forward network.

\begin{table}
  \caption{Error rates~(\%) and wall-clock times on the ImageNet dataset. The wall-clock time is the total time taken to train a given network by $200$ epochs.
  Relative speed-up is measured according to~\cref{RS}.}
  \label{vggcomp}
  \centering
  \resizebox{\linewidth}{!}{
  \begin{tabular}{lcccc}
    \toprule
    Network & Parameters & Error rate~(\%) & Wall-clock time~(h:m:s)& RS~(\%) \\
    \midrule
    VGG-16 & 138.4M & 32.30 & 191:17:00 & 0.0 \\
    Data Parallel VGG-16-4 & 553.6M & 27.73 & 142:56:18 & 25.3 \\
    Parareal VGG-16-4&147.5M&30.97&188:22:18& 1.5 \\
    \bottomrule
  \end{tabular}
  }
\end{table}

\Cref{vggcomp} shows that the wall-clock time of Data Parallel VGG-16-4 is the shortest.
This is because Data Parallel VGG-16-4 has very small number of layers, which takes a short time for communication to average the gradients for updating parameters.
On the other hand, in the case of Parareal VGG-16-4, as shown in \Cref{gradvgg}, the computation time of the parallel subnetwork is very small, and the majority of the total computation time is the computation time of the coarse network.
Therefore, Parareal VGG-16-4 has a parallel structure, but the effect of parallelization is not significant.
In fact, the parareal algorithm works more effectively in networks with a large number of layers, but maintains accuracy and does not slow down even in networks with a small number of layers.

% Subsection: ResNet-1001
\subsection{ResNet-1001}
\begin{figure}
 \centering
 \includegraphics[width=0.8\linewidth]{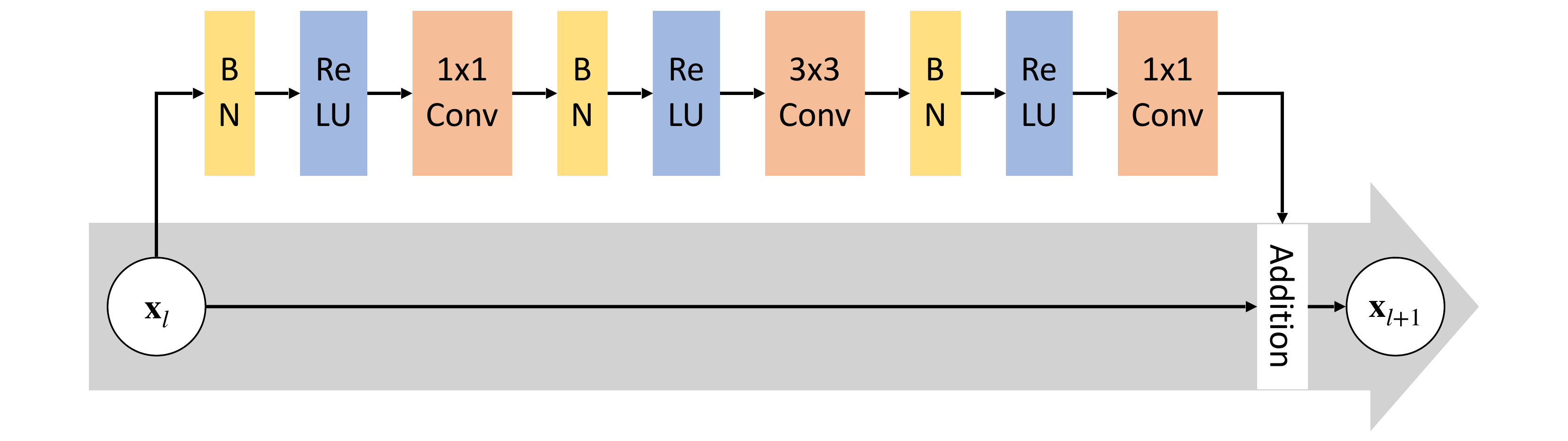}
 \caption{Bottleneck structure of an RU used in ResNet-1001. The first $1 \times 1$ convolution squeezes the number of channels of an input and the last $1 \times 1$ convolution increases the number of channels of an intermediate result.
 If $\bx_l$ and $\bx_{l+1}$ have different numbers of channels, then a $1 \times 1$ convolution is added to the skip connection in order to match the number of channels.}
 \label{ru}
\end{figure}

\begin{table}
 \caption{Architecture of ResNet-1001 for the dataset CIFAR-100.
 Downsampling with stride $2$ is performed in the second and third stages of the block-repetitive structure.
 A layer consisting of an $n \times n$ convolution with $k$-channel ouput, an average pooling with output size $1 \times 1$, and a $k$-way fully connected layer are denoted by [$n \times n, k$], avgpool, and $k$-d~FC, respectively.}
 \label{resnet}
 \centering
 \begin{tabular}{c|c|c}
 \toprule
 Layer & Output size & ResNet-1001 \\
 \midrule
 Preprocessing & $32 \times 32$ & [$3 \times 3$, $16$] \\ \hline
\multirow{3}{*}{\begin{tabular}[c]{@{}c@{}}\\~\\ Block-repetitive\\ substructure\end{tabular}} & $32 \times 32$ & $\begin{bmatrix} 1 \times 1,16 \\ 3 \times 3,16 \\ 1 \times 1,64 \end{bmatrix} \times 111$ \\ \cline{2-3}
 & $16 \times 16$ & $\begin{bmatrix} 1 \times 1,32 \\ 3 \times 3,32 \\ 1 \times 1,128 \end{bmatrix} \times 111$ \\ \cline{2-3}
 & $8 \times 8$ & $\begin{bmatrix} 1 \times 1,64 \\ 3 \times 3,64 \\ 1 \times 1,256 \end{bmatrix} \times 111$ \\ \hline
 Postprocessing & $1 \times 1$ & avgpool $+$ 100-d FC \\
 \bottomrule
 \end{tabular}
 \end{table}
Next, we will apply the proposed parareal neural network to ResNet~\cite{he2016deep,he2016identity}, which is typically one of the very deep neural networks.
In ResNet, an RU is constructed by adding a skip connection, i.e., it is written as
\begin{equation*}
\mathbf{x}_l = H_l(\mathbf{x}_{l-1}) + \mathbf{x}_{l-1}.
\end{equation*}
\Cref{ru} shows the structure of an RU, called bottleneck used in ResNet.
In particular, we deal with ResNet-1001 for the CIFAR dataset~\cite{krizhevsky2009cifar}, where $1,001$ is the number of layers.
ResNet-1001 consists of a single $3\times 3$ convolution layer, a sequence of 333~RUs with varying feature map dimensions, and a global average pooling followed by a fully connected layer.
\Cref{resnet} shows the details of the architecture of ResNet-1001 for the CIFAR-100 dataset.

As presented in \Cref{Sec:PNN}, ResNet-1001 can be decomposed as follows: a preprocessing operator $C_{\delta}$:~$X \rightarrow W_0$ as the $3 \times 3$ convolution layer, a block-repetitive substructure $g_{\phi}$:~$W_0 \rightarrow W_1$ as 333 RUs, and a postprocessing operator $h_{\varepsilon}$:~$W_1 \rightarrow Y$ consisting of the global average pooling and the fully connected layer.
More specifically, we have $X = \mathbb{R}^{3 \times 32 \times 32}$, $W_0 = \mathbb{R}^{16 \times 32 \times 32}$, $W_1 = \mathbb{R}^{256 \times 8 \times 8}$, and $Y = \mathbb{R}^m$ where $m$ is the number of classes of images.

\subsubsection{Parareal transformation}
The design of a parareal neural network with $N$ parallel subnetworks for ResNet-1001, denoted as \textit{Parareal ResNet-$N$}, can be completed by specifying the structures $g_{\phi_j}^j$, $C_{\delta_j}^j$, and $F_{\eta_j}^j$.
For convenience, the original neural network ResNet-1001 is called Parareal ResNet-1.
We assume that $N = 3N_0$ for some positive integer $N_0$.
We observe that $g_{\phi}$ can be decomposed as
\begin{equation*}
g_{\phi} = g_{\phi_{N}}^{N} \circ \dots \circ g_{\phi_{2N_0+1}}^{2N_0+1} \circ g_{\phi_{2N_0}}^{2N_0} \circ \dots \circ g_{\phi_{N_0+1}}^{N_0+1} \circ g_{\phi_{N_0}}^{N_0} \circ \dots \circ g_{\phi_{1}}^{1},
\end{equation*}
where each of $g_{\phi_{j}}^{j}$:~$X_{j-1} \rightarrow X_{j}$ consists of $\lceil 333/N \rceil$~RUs with
\begin{align*}
X_{j} = \begin{cases}\mathbb{R}^{64 \times 32 \times 32} & \textrm{ for } j=1,\dots,N_0,\\ \mathbb{R}^{128 \times 16 \times 16} & \textrm{ for } j=N_0+1,\dots,2N_0,\\ \mathbb{R}^{256 \times 8 \times 8} & \textrm{ for } j=2N_0+1,\dots,N,\end{cases} \quad \phi = \bigoplus_{j=1}^N \phi_j.
\end{align*}
Similarly to the case of VGG-16, the preprocessing operators are defined as follows: $C_{\delta_1}^1 = C_{\delta}$ and $C_{\delta_j}^j$ for $j>1$ consists of a $1 \times 1$ convolution to match the number of channels after appropriate number of $3 \times 3$ max pooling layers with stride $2$ to match the image size.
For the coarse network, we first define a coarse RU consisting of two $3 \times 3$ convolutions and skip-connection.
If the downsampling is needed, then the stride of first convolution in coarse RU is set to $2$.
We want to define $F_{\eta_{j}}^j$:~$X_j \rightarrow X_{j+1}$ having smaller number of (coarse) RUs than $g_{\phi_{j+1}}^{j+1}$ but a similar coverage to $g_{\phi_{j+1}}^{j+1}$.
Note that the receptive field of $g_{\phi_j}^j$ covers the input size $32 \times 32$.
In terms of receptive field, even if we construct $F_{\eta_j}^j$ with fewer coarse RUs than $\lceil 333/N \rceil$, it can have similar coverage to the parallel subnetwork~$g_{\phi_{j}}^{j}$.

\begin{table}
  \caption{Error rates~(\%) on the CIFAR-100 dataset of Parareal ResNet-3, where $N_c$ is the number of coarse RUs in each component $F_{\eta_j}^j$ of the coarse network.}
  \label{nc}
  \centering
  \begin{tabular}{c|c}
    \toprule
    $N_c$   & Error rate~(\%) \\
    \midrule
    1 & 23.47 \\
    2 & 22.20 \\
    4 & 21.14 \\
    8 & 20.85 \\
    16& 20.83  \\
    \midrule
    \multicolumn{2}{l}{Reference~(ResNet-1001): 21.13} \\
    \bottomrule
  \end{tabular}
\end{table}

Let $N_c$ be the number of coarse RUs in $F_{\eta_j}^j$ of the coarse network.
Here, we present how $N_c$ affects the performance of the parareal neural network.
Specifically, we experimented with Parareal ResNet-3 with $g_{\phi_j}^j$ consisting of $111$ RUs.
\Cref{nc} shows the error rates of Parareal ResNet-3 with respect to various $N_c$.
The performance of Parareal ResNet-3 with $N_c = 4$ is similar to ResNet-1001.
In fact, four coarse RUs consist of one $3 \times 3$ convolution with stride $2$ and seven $3 \times 3$ convolutions
so that the receptive field covers $31 \times 31$ pixels, which is almost all pixels of the input $\bx \in X$.
Parareal ResNet-3 shows better performance than ResNet-1001 when $N_c > 4$.
Since Parareal ResNet-3 with $N_c = 4$ shows similar performance to ResNet-1001, we may say that $1$ RUs in $F_{\eta_j}^j$ can approximate $111/4 \approx 28$ RUs in $g_{\phi_j}^j$.
Generally, if we use $N$ parallel subnetworks $(N \ge 3)$, each $333/N$~RUs in $g_{\phi_j}^j$ can be approximated by the $N_c$ RUs in $F_{\eta_j}^j$ whenever we select $N_c=\lceil 12/N \rceil$.

\subsubsection{Numerical results}
With fixed $N_c=\lceil 12/N \rceil$, we report the classification results of Parareal ResNet with respect to various $N$ on datasets CIFAR-10, CIFAR-100, SVHN, and MNIST.
Decay of the training loss of Parareal ResNet-$N$~($N=1,3,6,12,18$) for various datasets is depicted in~\Cref{logloss}.
As shown in~\Cref{logloss}, the training loss converges to a smaller value for larger $N$.
It seems that such a phenomenon is due to the increase of the number of parameters in Parareal ResNet when $N$ increases.
In the cases of MNIST and SVHN, oscillations of the training loss are observed.
It is well-known that such oscillations are caused by excessive weight decay and can be removed by dropout~\cite{zagoruyko2016wide}.
\begin{figure}
  \centering
  \includegraphics[width=0.8\linewidth]{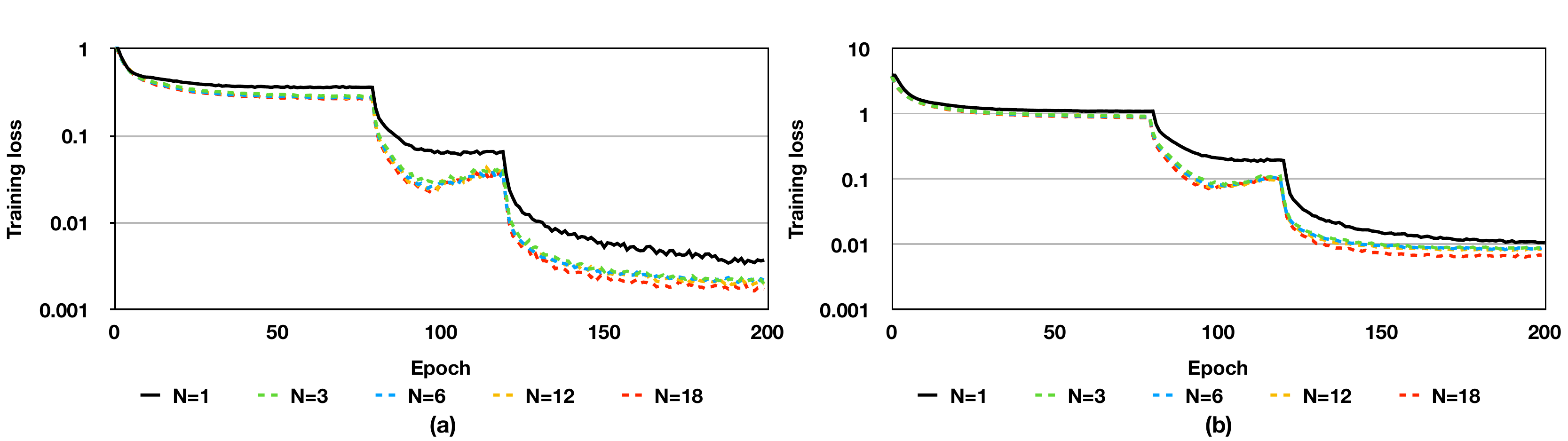}\\
  \includegraphics[width=0.8\linewidth]{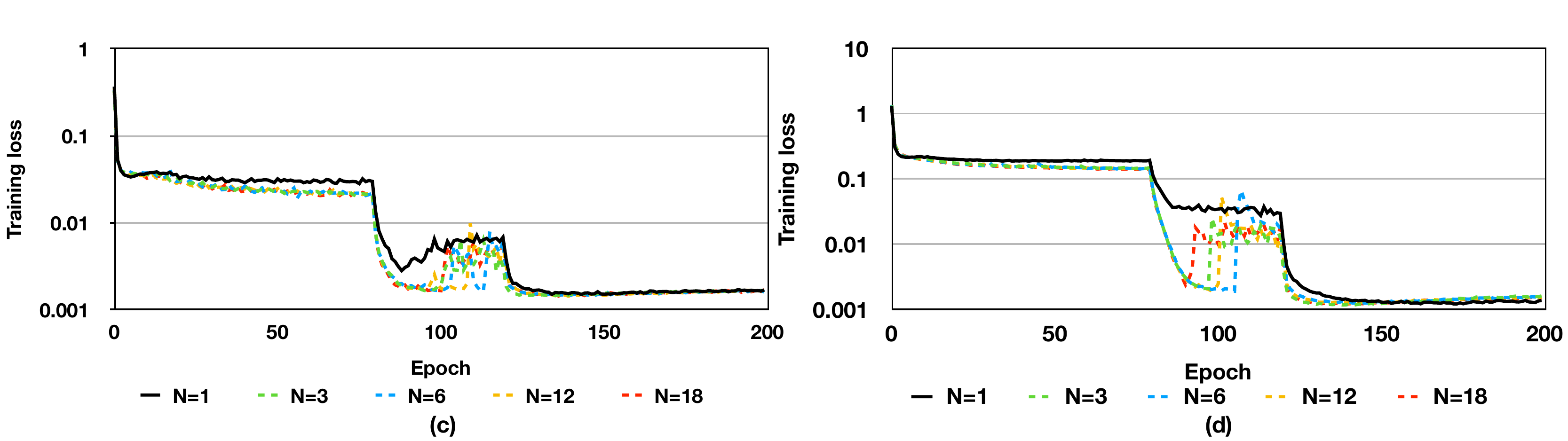}
  \caption{Comparison of the training loss for Parareal ResNet-$N$~($N=1,3,6,12,18$) on various datasets: \textbf{(a)} CIFAR-10, \textbf{(b)} CIFAR-100, \textbf{(c)} MNIST, and \textbf{(d)} SVHN results.}
  \label{logloss}
\end{figure}

We next study how the classification performance is affected by $N$, the number of subnetworks.
\Cref{errorrate} shows that the error rates of Parareal ResNet-$N$ are usually smaller than ResNet-1001.
There are some exceptional cases that the error rate of Parareal ResNet-$N$ exceeds that of ResNet-1001: $N=6,12$ and $18$ for SVHN.
It is known that these cases occur when there are oscillations in the decay of the training loss~\cite{zhang2018three}; see~\Cref{logloss}.
As we mentioned above, such oscillations can be handled by dropout.

\begin{table}
  \caption{Error rates~(\%) on the CIFAR-10, CIFAR-100, MNIST, and SVHN datasets of Parareal~ResNet-$N$ ($N=1,3,6,12,18$) with $N_c = \lceil 12/N \rceil$.}
  \label{errorrate}
  \centering
  \resizebox{\linewidth}{!}{
  \begin{tabular}{cccccccc}
    \toprule
    $N$ &\begin{tabular}{c}Subnetwork \\ Parameters \end{tabular}&\begin{tabular}{c}Coarse network \\ Parameters \end{tabular}& \begin{tabular}{c}Total \\ Parameters \end{tabular}& CIFAR-10 & CIFAR-100 & MNIST & SVHN \\
    \midrule
    1 &-&-&10.3M& 4.96 & 21.13 & 0.34 & 3.17 \\ \hline
    3 &3.4M& 5.6M& 15.9M& 4.61 & 21.14 & 0.31 & 3.11 \\
    6 &1.7M&5.7M& 16.1M & 4.20 & 20.87 & 0.31 & 3.21\\
    12 &0.9M& 5.8M&16.2M & 4.37 & 20.42 & 0.28 & 3.25\\
    18 &0.6M&8.9M& 19.4M & 4.02 & 20.40 & 0.33 & 3.29\\
    \bottomrule
  \end{tabular}}
\end{table}

Like the case of VGG-16, we investigate the elapsed time for forward and backward propagations of parareal neural networks.
\Cref{grad} shows the virtual wall-clock time for forward and backward computation of Parareal ResNet-$N$ with various $N$ for the input $\bx \in \mathbb{R}^{3 \times 32 \times 32}$.
As shown in~\Cref{grad}, the larger $N$, the shorter the computing time of the parallel subnetworks $g_{\phi_j}^j$, while the longer the computing time of the coarse network.
This is because as $N$ increases, the depth of each parallel subnetwork $g_{\phi_j}^j$ becomes shallower while the number of $F_{\eta_j}^j$ in the coarse network increases.
On the other hand, each preprocessing operator $C_{\delta_j}^j$ is designed to be the same as or similar to the preprocessing operator $C_\delta$ of the original neural network and the postprocessing operator $h_\varepsilon$ is the same as the original one.
Therefore, the computation time for the pre- and postprocessing operators does not increase even as $N$ increases.
As a trade-off between the decrease in time in parallel subnetworks and the increase in time in the coarse network, we observe in~\Cref{grad} that the elapsed time of forward and backward computation decreases as $N$ increases when $N \leq 18$.

\begin{table}
  \caption{Forward/backward computation time for Parareal ResNet-$N$~($N=1,3,6,12,18$).
  The time is measured in one iteration for CIFAR-100 dataset input $\bx \in \mathbb{R}^{3 \times 32 \times 32}$ with batch size $128$.
  }
  \label{grad}
  \centering
  \begin{tabular}{cccccc}
    \toprule
    \multicolumn{6}{c}{Virtual wall-clock time~(ms)}\\
    \cmidrule(r){1-6}
    $N$ & Preprocessing     & \begin{tabular}{c}Parallel\\subnetworks\end{tabular}    & \begin{tabular}{c}Coarse\\network\end{tabular} & Postprocessing  & Total \\
    \midrule
    1&0.25/6.46&	443.81/1387.62&	-&	0.06/3.18&	444.12/1397.26\\
    \midrule
    3&0.25/6.45&	131.92/458.87&	10.01/97.60&	0.06/3.71&	142.24/566.63\\
    6&0.27/6.42&	67.59/219.72&	14.68/137.08&	0.06/3.33&	82.60/366.55\\
    12&0.28/6.59&	48.47/113.33&	17.97/149.52&	0.06/3.63&	66.78/273.07\\
    18&0.29/6.17&	30.40/77.84&	27.87/163.25&	0.06/3.64&	58.62/250.90\\
    24&0.29/6.58&	22.71/58.04&	41.03/242.87&	0.06/3.54&	64.09/311.03\\
    \bottomrule
  \end{tabular}
\end{table}

\begin{remark}
In our experiments, the decomposition of the block-repetitive structure of ResNet-1001 was done so that each parallel subnetwork of Parareal ResNet had the same number of RUs.
However, such a decomposition is not optimal because each RU has the different computational cost.
It is expected that the uniform decomposition of the Parareal ResNet in terms of computational cost will further improve the parallel efficiency reducing virtual wall-clock time even above $N = 18$.
\end{remark}

Finally, we compare the performance of the Parareal ResNet to two existing approaches of parallel computing for neural networks: WarpNet~\cite{fok2018decoupling} and data parallelism.
WarpNet replaces every $K$ RUs in ResNet by a \textit{$K$-warp operator}, which can be evaluated using RUs arranged in a parallel manner and their derivatives.
For example, we display the parallel structure of the $2$-warp operator in~\Cref{warp}; see~\cite{fok2018decoupling} for further details.
In what follows, Data Parallel ResNet-$N$ and WarpNet-$N$ denote the data parallelized ResNet-1001 with $N$ GPUs and WarpNet with $(N-1)$-warp operators using $N$ GPUs, respectively.
We compare the error rate and the wall-clock time of each parallelized network with the CIFAR-100 dataset.
\Cref{comparison} shows that only RS of Parareal ResNet is greater than 0.
This is because when the number of layers of WarpNet and Data Parallel ResNet is very large, data communication between GPUs is required, which can frequently cause communication bottlenecks at each layer.
On the other hand, in the forward propagation of Parareal ResNet, only a single communication process among GPUs is needed: communication between each parallel subnetwork and the coarse network.
Therefore, Parareal ResNet is relatively free from communication bottlenecks compared to the other methods.
In conclusion, Parareal ResNet outperforms the other two methods in parallelization of deep neural networks in terms of training time reduction.
\begin{figure}
\centering
\includegraphics[width=0.7\linewidth]{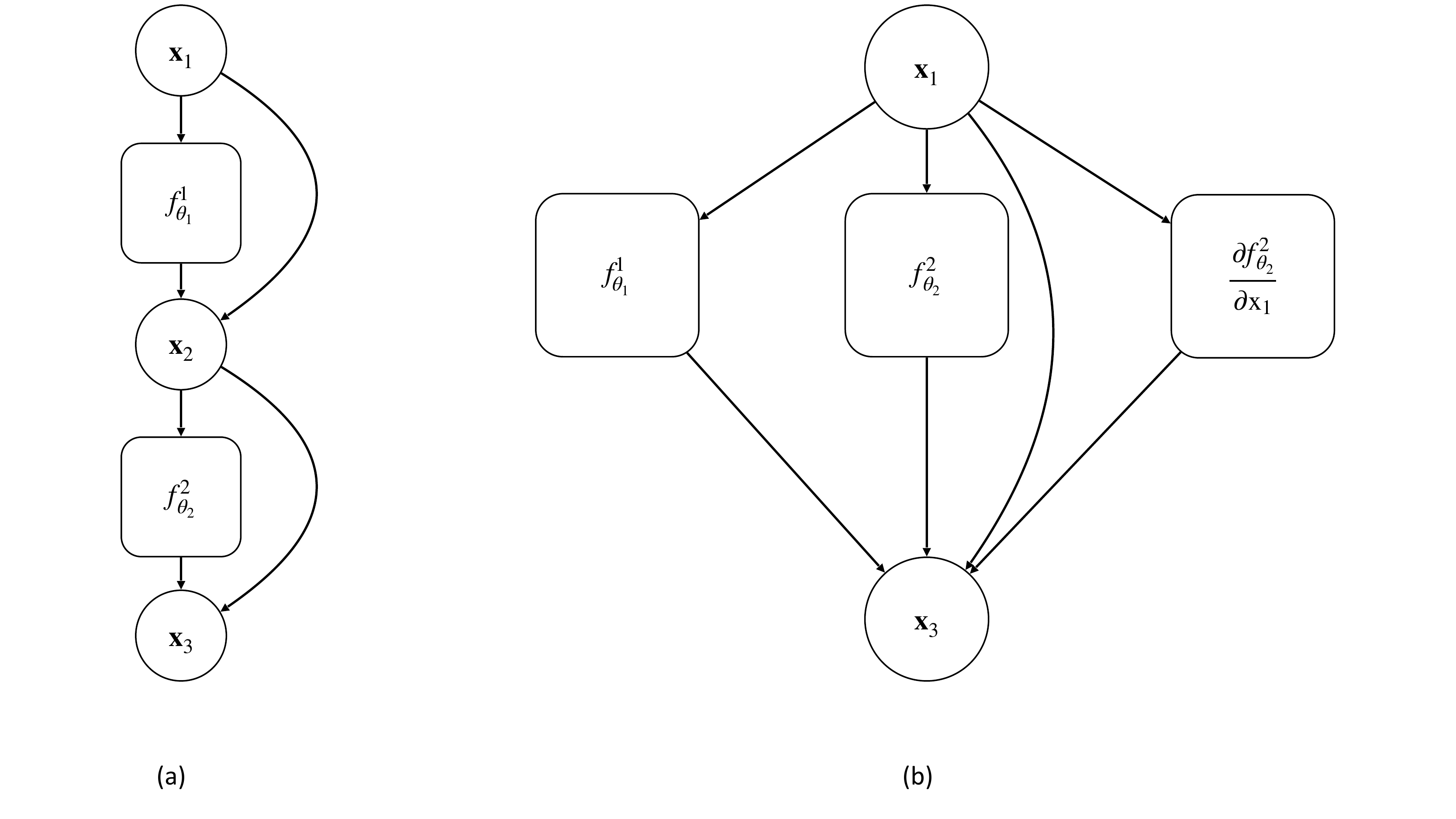}
\caption{Structure of the $K$-warp operator, $K=2$: \textbf{(a)} Two RUs $f_{\theta_1}^1$ and $f_{\theta_2}^2$ in ResNet. \textbf{(b)}~The corresponding $2$-warp operator in WarpNet.}
\label{warp}
\end{figure}
\begin{table}
  \caption{Error rates~(\%) and wall-clock times on the CIFAR-100 dataset. The wall-clock time is the total time taken to train a given network by $200$ epochs.}
  \label{comparison}
  \centering
  \resizebox{\linewidth}{!}{
  \begin{tabular}{lcccc}
    \toprule
    Network & Parameters & Error rate~(\%) & Wall-clock time~(h:m:s)& RS~(\%) \\
    \midrule
    ResNet-1001 & 10.3M & 21.13 & 22:44:53 & 0.0\\
    \midrule
    Data Parallel ResNet-3& 30.9M & 21.09 & 46:50:20 & -105.9\\
    Parareal ResNet-3& 15.9M& 21.14 & 16:28:38 & 27.6\\
    WarpNet-3 & 15.3M & 19.95 & 90:56:46 & -299.8\\
    \midrule
    Data Parallel ResNet-6& 61.8M & 20.42 & 68:10:06 & -199.7\\
    Parareal ResNet-6& 16.1M& 20.87 & 11:48:13 & 48.1\\
    \bottomrule
  \end{tabular}
  }
\end{table}

\begin{remark}
  In terms of memory efficiency, Data Parallel VGG-16 and Data Parallel ResNet allocate all parameters to each GPU, which is a waste of memory of GPUs.
  WarpNet-$N$ requires duplicates of parameters since it computes RUs and their derivatives simultaneously at forward propagation so that the memory efficiency is deteriorated.
  On the other hand, in Parareal VGG-16 and Parareal ResNet, the entire model can be equally distributed to each GPU as a parallel subnetwork; only a single GPU requires additional memory to store the parameters of the coarse network which is much smaller than the entire model.
  Therefore, the proposed Parareal neural network outperforms the other two methods in the terms of memory efficiency as well.
\end{remark}

%%%%%%%%%%%%%%%%%%%%%%%%%%%%%%%%%%%%%%%%%%%%%%%%%%%%%%%%%%%%%%%%%%%%%%
\section{Conclusion}
\label{Sec:Conclu}
In this paper, we proposed a novel methodology to construct a parallel neural network called the parareal neural network, which is suitable for parallel computation using multiple GPUs from a given feed-forward neural network.
Motivated by the parareal algorithm for time-dependent differential equations, the block-repetitive part of the original neural network was partitioned into small pieces to form parallel subnetworks of the parareal neural network.
The coarse network that corrects differences at the interfaces among subnetworks was introduced so that the performance of the resulting parareal network agrees with the original network.
As a concrete example, we presented how to design the parareal neural network corresponding to VGG-16 and ResNet-1001.
Numerical results were provided to highlight the robustness and parallel efficiency of Parareal VGG-16 and Parareal ResNet.
To the best of our knowledge, the proposed methodology is a new kind of multi-GPU parallelism in the field of deep learning.

\section*{Acknowledgments}
The authors would like to thank Prof. Ganguk Hwang for his insightful comments.

\bibliographystyle{siamplain}
\bibliography{refs_Parareal_NN}

\end{document}